\documentclass[journal]{IEEEtran}
\usepackage{amsmath,amsfonts}
\allowdisplaybreaks
\usepackage{amsthm}
\usepackage[ruled]{algorithm2e}
\usepackage{cite}
\usepackage{booktabs}
\usepackage{float}
\usepackage{array}
\usepackage[caption=false,font=normalsize,labelfont=sf,textfont=sf]{subfig}
\usepackage{textcomp}
\usepackage{stfloats}
\usepackage{url}

\usepackage{xcolor}
\usepackage{verbatim}
\usepackage{graphicx}
\def\BibTeX{{\rm B\kern-.05em{\sc i\kern-.025em b}\kern-.08em
    T\kern-.1667em\lower.7ex\hbox{E}\kern-.125emX}}
\usepackage{balance}

\newtheorem{assumption}{Assumption}
\newtheorem{proposition}{Proposition}
\newtheorem{lemma}{Lemma}
\newtheorem{remark}{Remark}

\usepackage{comment}
\newboolean{showcomments}
\setboolean{showcomments}{true}

\newcommand{\fliu}[1]{\ifthenelse{\boolean{showcomments}}
        { \textcolor{red}{(FL:  #1)}}{}}
\newcommand{\lth}[1]{\ifthenelse{\boolean{showcomments}}
        { \textcolor{blue}{(THL:  #1)}}{}} 

\IEEEaftertitletext{%
  \vspace{-3.5\baselineskip}%
}

\begin{document}

\title{A Non-Iterative Algorithm for Clearing Two-Layer Energy-Sharing Markets with Voltage Constraints}

\author{\IEEEauthorblockN{Tonghua Liu$^{1}$, Yifan Su$^{2}$, Member, IEEE, Zhaojian Wang$^{3}$, Member, IEEE, and Feng Liu$^{1,*}$, Senior Member, IEEE}\\[-0pt]
\IEEEauthorblockA{$^{1}$Department of Electrical Engineering, Tsinghua University, Beijing, China\\[-0pt]
$^{2}$School of Electrical Engineering, Chongqing University, Chongqing, China\\[-0pt]
$^{3}$School of Automation and Intelligent Sensing, Shanghai Jiao Tong University, Shanghai, China\\[-0pt]
liuth21@mails.tsinghua.edu.cn, *lfeng@tsinghua.edu.cn}
}

\maketitle

\begin{abstract}
Real-time hierarchical energy-sharing markets are promising to coordinate large numbers of prosumers. Still, most existing clearing methods rely on linearized or DC power-flow models and do not explicitly handle reactive power or voltage-security constraints. With AC network constraints, the problem becomes a large-scale bilevel Mathematical Program with Equilibrium Constraints (MPEC) that is difficult to solve in real time. This paper develops a non-iterative clearing algorithm for two-layer energy-sharing markets with voltage constraints. We first derive an efficient best-response function for each lower-layer energy-sharing market and reduce the equilibrium search to one dimension by exploiting the pricing-coupling structure. We then embed this function into the upper-layer network-constrained problem and reformulate the bilevel MPEC as a single-level mixed-integer second-order cone program (MISOCP), which is computationally tractable. Case studies on the IEEE 123-bus system with 12,300 prosumers show that the proposed method preserves nodal voltages within prescribed limits and delivers solutions with maximum errors below 0.01\% in 0.829 s. 

\end{abstract}


\section{Introduction}


The transition toward carbon neutrality has accelerated the deployment of distributed energy resources (DERs) in distribution networks, transforming end-users from passive consumers into active prosumers \cite{sotoPeertopeerEnergyTrading2021,guerreroDecentralizedP2PEnergy2019,wuSharingEconomyLocal2023}. To coordinate these resources more efficiently, Peer-to-Peer (P2P) energy sharing and local energy markets have emerged as promising mechanisms for enabling direct transactions among prosumers and improving local renewable-energy utilization, economic efficiency, and system flexibility \cite{chenReviewEnergySharing2022,zhouStateoftheArtAnalysisPerspectives2020,zedanReviewPeertopeerEnergy2024}.
Despite this promise, large-scale energy-sharing markets remain difficult to clear in a way that is both computationally scalable and physically feasible. 

The main challenge is that market outcomes must respect distribution network constraints, especially voltage security limits, while simultaneously accommodating massive prosumer participation \cite{guerreroDecentralizedP2PEnergy2019,kleinjanVoltageSupportBased2022}. This challenge becomes particularly acute when nonlinear alternating-current (AC) network constraints are considered, primarily motivating this work.

Existing studies can be cast into three categories. The first category incorporates network constraints directly into local energy markets. Representative approaches combine branch-flow models, second-order cone programming (SOCP) relaxations, Distribution Locational Marginal Pricing (DLMP)-based pricing, and network-aware transactive coordination to mitigate congestion and voltage violations \cite{farivarBranchFlowModel2013,golambahriDecentralizedEnergyFlexibility2025,wuOptimalManagementTransactive2020,yanPeertoPeerTransactiveEnergy2023,yanDistributionNetworkConstrainedOptimization2021,yanDistributedCoordinationCharging2023}. While these methods improve physical realism, they are typically tested on relatively small systems and do not fully address scalability to very large prosumer populations.

The second category focuses on computational scalability. Centralized optimization can, in principle, provide globally optimal clearing results, but it suffers from privacy concerns and severe growth in dimensionality in large-scale settings \cite{udabalaAggregatedEnergyInteraction2025}. Distributed approaches, including game-theoretic methods, decentralized market mechanisms, and Alternating Direction Method of Multipliers (ADMM)-based decompositions, offer several more scalable alternatives \cite{yaagoubiEnergyTradingSmart2017,sotoPeertopeerEnergyTrading2021,guerreroDecentralizedP2PEnergy2019}. However, once hard physical coupling (e.g. AC power flow) is enforced, especially nodal voltage constraints, iterative distributed algorithms may converge slowly and become difficult to tune in practice \cite{noorpiADMMbasedOptimumPower2020,louADMMBasedDistributed2024,baekSecurityConstrainedP2PEnergy2024,dengParallelMultiBlockADMM2017,liuDistributedVoltageControl2018,fengPeertoPeerEnergyTrading2023,dongDSOprosumersDuallayerGame2024}.

The third category seeks to combine scalability and physical feasibility through hierarchical or bilevel market architectures \cite{linNestedbilevelEnergy2023,hongbilevelGametheoreticDecisionmaking2023,zarabieFairnessRegularizedDLMPBasedbilevel2019,noorfatimabilevelPeertoPeerEnergy2024}. These frameworks preserve the interaction between local market behavior and network constraints. However, they usually lead to large Mathematical Programs with Equilibrium Constraints (MPECs) that still require extensive iterative computation. Su et al. \cite{suSharingEnergyWider2025} proposed a highly scalable two-layer market architecture with multiple lower-layer energy-sharing markets (L-ESMs) coordinated by an upper-layer energy-sharing market (U-ESM). This framework, nevertheless, relies on a linearized power-flow model and therefore does not explicitly enforce voltage constraints. Xia et al. \cite{xiaNonIterativeDecentralizedPeertoPeer2024} moved toward non-iterative market clearing under network constraints. Still, the required power-flow linearization and network reduction may reduce model fidelity in large-scale, heterogeneous distribution networks with a massive number of prosumers.

The aforementioned observations highlight a clear gap in the literature. Existing methods do not simultaneously provide the scalability of a two-layer market architecture with massive prosumers, explicit voltage-security guarantees under AC power flow, and a high-efficiency clearing procedure suitable for real-time implementation. The core bottleneck is that combining massive prosumer participation with nonlinear network constraints yields a high-dimensional bilevel MPEC whose lower-level equilibrium structure is intractable at scale.

To address this challenge and fill this gap, this paper derives an explicit best-response function by unfolding the specific structure of the L-ESM solution, thereby making the lower-layer equilibrium tractable and enabling non-iterative clearing with voltage constraints. The main contributions are as follows:
\begin{enumerate}
    \item \textbf{Efficient algorithms for determining the best-response functions of L-ESMs.} Building upon the established equivalence between the Nash equilibrium in an L-ESM and a convex optimization problem \cite{suSharingEnergyWider2025}, we analyze the solution structure to develop efficient algorithms that determine the best-response function. 
    Unlike standard multi-parametric programming, which requires a high-dimensional active-set search \cite{guptaNovelApproachMultiparametric2011}, we exploit the L-ESM's specific pricing coupling to collapse the high-dimensional search into a simple 1-dimensional search, thereby remarkably reducing the computational burden. This function serves as a computable surrogate for the complex equilibrium constraints of the L-ESM, eliminating the need for iterative equilibrium search and significantly reducing the computational complexity.

    \item \textbf{A high-efficiency non-iterative market-clearing algorithm with voltage constraints.}
    By integrating the determined best-response function into the U-ESM and employing a Second-Order Cone Program (SOCP) relaxation of the branch flow model, we transform the original bilevel MPEC into a single-level, tractable Mixed-Integer SOCP (MISOCP). It yields a non-iterative market-clearing algorithm that bridges the critical trade-off between computational efficiency and voltage limits, ensuring network feasibility for large-scale applications.
\end{enumerate}


The rest of the paper is organized as follows. Section II introduces the market model and bilevel formulation. Section III analyzes the L-ESM subproblem. Section IV develops the best-response function and the non-iterative reformulation. Section V reports case studies. Section VI concludes the paper.

\section{Market Description and Formulation}

This work builds on the two-layer hierarchical market architecture proposed by Su et al. \cite{suSharingEnergyWider2025} and further 
enforces additional voltage constraints. This architecture 
naturally decouples the prosumer-level trading details from the 
network-level operation, thereby providing a tractable framework to 
enforce rigorous AC power flow and voltage constraints at the upper 
level. To this end, we replace the original \emph{linearized power flow model} with the \emph{AC branch flow model}, yielding a nonlinear bilevel MPEC with second-order cone constraints, which is challenging to solve. 

As shown in Fig. \ref{fig:market_arch}, the market comprises a number of prosumers, multiple L-ESMs, and a single U-ESM. The L-ESMs handle local trading among prosumers but do not necessarily achieve a full balance, whereas the U-ESM performs global clearing and enforces network constraints.

\begin{figure}[h]
    \centering
\includegraphics[width=1\linewidth]{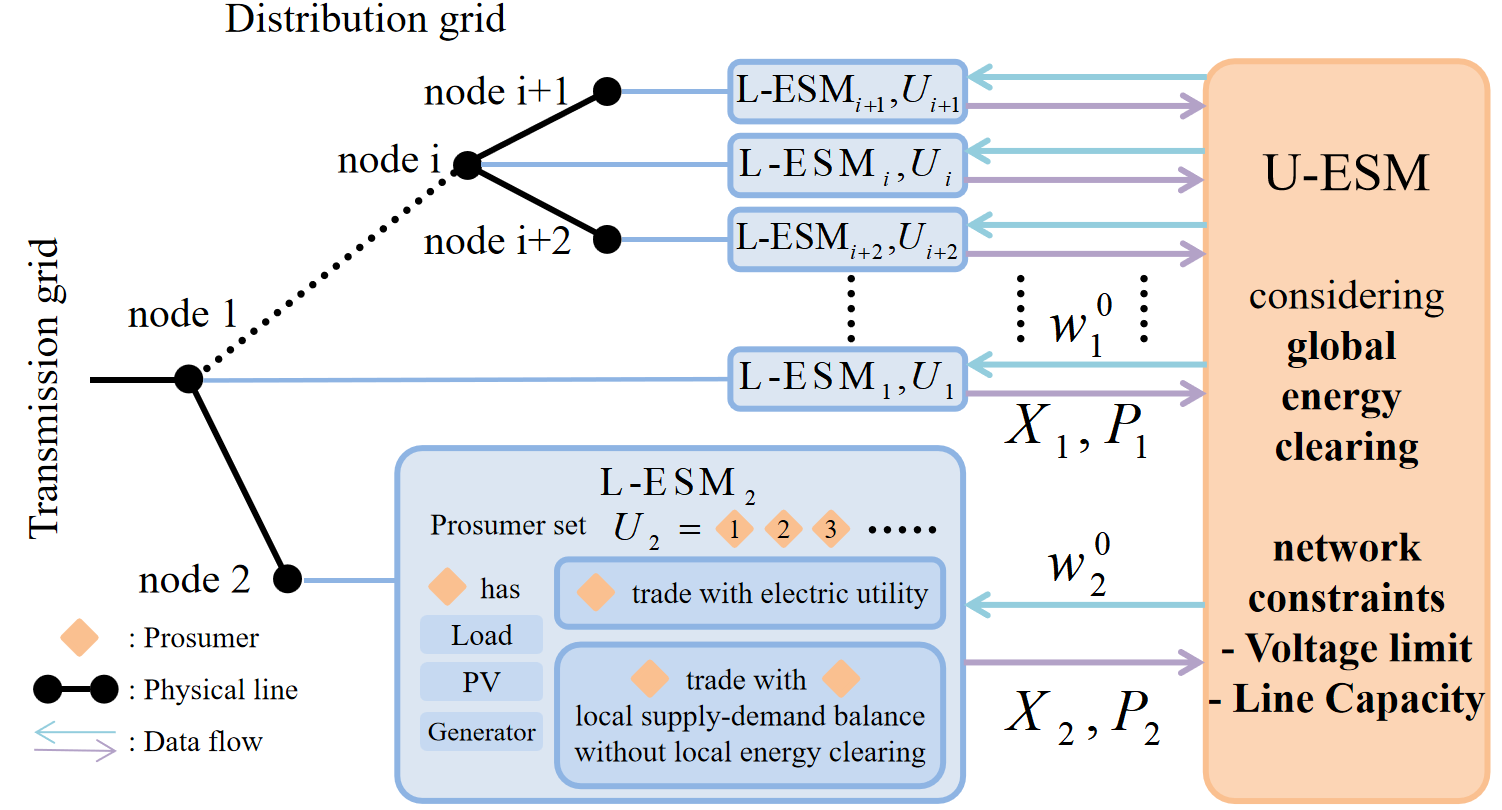}
    \caption{The two-layer hierarchical market and its clearing scheme}
    \label{fig:market_arch}
\end{figure}

\textbf{Prosumers:} Each prosumer is equipped with a load, a non-adjustable power generation device, such as photovoltaic panels, and an adjustable power generation device, such as a gas generator. Prosumers participate in L-ESMs or trade with the electric utility, seeking to minimize their individual costs, which consist of power generation costs and electricity trading fees. The set of prosumers that join L-ESM $i$ is denoted by $U_i$. 

\textbf{L-ESMs:} The L-ESMs, indexed by $i \in \mathcal{N}$, manage the
lower-layer trading among a cluster of local prosumers $U_i$. The L-ESM adjusts the local energy-sharing price around a base price, relying on the local supply-demand balance. Notably, the L-ESM does not perform a full local
clearing; instead, it allows a certain unbalanced portion to spill over to the U-ESM.

\textbf{U-ESM:} The U-ESM coordinates all L-ESMs 
to enable wider-area energy sharing. Operated as a public utility, it clears the market to minimize the total network loss subject to line limits and voltage constraints. The clearing price is delivered to individual L-ESMs to coordinate their local trading. 

The overall trading and market-clearing can be formulated as a Stackelberg game below:
\begin{subequations}\label{model:simple_stage_raw}
    \begin{align}
    & \min_{w_i^0, P_i, Q_i, P_{ij}, Q_{ij}, l_{ij}, v_i} \quad \sum_{(i,j)\in \mathcal{E}} r_{ij}l_{ij} \label{obj:USM}\\
    & \text{s.t.} \notag \\
    & \quad P_j = \sum_{k:j \to k} P_{jk} - \sum_{i:i \to j} (P_{ij} - r_{ij}l_{ij}), \quad \forall j \in \mathcal{N} \label{cons:node_P} \\
    & \quad Q_j = \sum_{k:j \to k} Q_{jk} - \sum_{i:i \to j} (Q_{ij} - x_{ij}l_{ij}), \quad \forall j \in \mathcal{N} \label{cons:node_Q}\\
    & \quad v_j = v_i - 2(r_{ij}P_{ij}+x_{ij}Q_{ij}) \notag \\
    & \quad \quad + (r^2_{ij} + x^2_{ij})l_{ij}, \quad \forall (i,j) \in \mathcal{E} \label{cons:node_v} \\
    & \quad P_{ij}^2 + Q_{ij}^2 = l_{ij}v_{i}, \quad \forall (i,j) \in \mathcal{E} \label{cons:line_limit}\\
    & \quad \underline{v_i} \leq v_i \leq \overline{v_i}, \quad \forall i \in \mathcal{N} \label{cons:v_lb_ub}\\
    & \quad 0 \leq l_{ij} \leq \overline{l_{ij}}, \quad \forall (i,j) \in \mathcal{E} \label{cons:l_lb_ub}\\
    & \quad \underline{Q_j} \leq Q_j \leq \overline{Q_j}, \quad \forall j \in \mathcal{N} \label{cons:Q_lb_ub}\\
    & \quad P_j = \sum_{m \in U_j}(x^*_m+p_m^{-*}-p_m^{+*}), \quad \forall j \in \mathcal{N} \label{cons:P_define}\\
    & \quad X_j = \sum_{m \in U_j}x^*_m, \quad \forall j \in \mathcal{N} \label{cons:X_define}\\
    & \quad \sum_{i \in \mathcal{N}}X_i = 0 \label{cons:sum_X}\\
    & \quad s_m^* = \{p_m^*, p_m^{+*}, p_m^{-*}, x_m^*\}, \forall m \in U_i, \forall i \in \mathcal{N} \notag \\
    & \quad \phantom{s_m^*} = \arg \min_{p_m,p_m^+,p_m^-,x_m} \frac{c_m}{2}p_m^2+b_mp_m+w^+_ip_m^+ \notag \\
    & \quad \phantom{s_m^* = \arg \min_{p_m,p_m^+,p_m^-,x_m}} - w_i^-p_m^--w_ix_m \label{obj:LSM_raw}\\
    & \quad \text{s.t.} \notag \\
    & \quad \quad D'_m +x_m+p^-_m = p_{solar,m}+p_m + p_m^+, \quad \forall m \in U_i \label{cons:energy_balance_raw}\\
    & \quad \quad w_i = w_i^0 - a_iX_i, \quad \forall i \in \mathcal{N} \label{cons:supply_demand}\\
    & \quad \quad 0 \leq p_m \leq \overline{p_m}, \quad \forall m \in U_i \label{cons:p_lb_ub}\\
    & \quad \quad p_m^+ \geq 0, p_m^- \geq 0, \quad \forall m \in U_i \label{cons:p_+_-}\\
    & \quad \quad p_m^+p_m^- = 0, \quad \forall m \in U_i \label{cons:su_relax}
    \end{align}
\end{subequations}
where $\mathcal{E}$ and $\mathcal{N}$ denote the sets of branches and nodes in the distribution network, respectively; 
$i, j \in \mathcal{N}$ are the node indices. 
For a branch $(i,j) \in \mathcal{E}$, $r_{ij}$ and $x_{ij}$ represent the resistance and reactance, 
while $P_{ij}$, $Q_{ij}$, and $l_{ij}$ are the active power flow, reactive power flow, 
and squared current magnitude, respectively. 
The symbols $\underline{v_i}$, $\overline{v_i}$, and $\overline{l_{ij}}$ represent the lower and upper bounds 
for voltage magnitude and the upper bound for squared current. 
At node $j$, $P_j$, $Q_j$, and $v_j$ denote the net injected active power, reactive power, and voltage magnitude. In the upper-layer market, $Q_j$ is modeled as the nodal reactive-power support available to the network operator, with bounds $\underline{Q_j}$ and $\overline{Q_j}$.
For the prosumers, $U_i$ denotes the set of prosumers in L-ESM $i$. 
For a specific prosumer $m \in U_i$, $x_m$ is the shared energy (positive for selling, negative for buying); 
$p_m$, $p_{solar,m}$, and $D'_m$ denote the adjustable generation, non-adjustable generation, and load, respectively. 
The parameters $c_m > 0$ and $0 < b_m < w_i^-$ are the cost coefficients of the adjustable generator, 
which has an upper bound $\overline{p_m}$. 
$p_m^+$ and $p_m^-$ denote the power purchased from and sold to the main grid at prices $w_i^+$ and $w_i^-$. 
The symbol $s_m^* = \{p_m^*, p_m^{+*}, p_m^{-*}, x_m^*\}$ denotes the optimal solution of the sub-problem. Accordingly, the set of optimal solutions for all prosumers in L-ESM $i$ is denoted as $S_i^* = \{s_m^*\}_{m \in U_i}$.
Finally, $w_i$ is the local sharing price, $w_i^0$ is the base price set by the U-ESM, 
$a_i > 0$ is the price elasticity coefficient regarding the uncleared energy, 
$P_i := \sum_{m \in U_i}(x^*_m+p_m^{-*}-p_m^{+*})$ denotes the total power exchange with the main grid, 
and $X_i := \sum_{m \in U_i} x^*_m$ denotes the uncleared sharing energy.

Based on Propositions 1 and 2 in \cite{suSharingEnergyWider2025}, the complementarity constraint \eqref{cons:su_relax} can be relaxed, and the Nash equilibrium of the L-ESM is equivalent to the solution of a convex optimization problem. 
This equivalence allows us to reformulate the objective function, thereby ensuring that it satisfies the pricing constraint \eqref{cons:supply_demand}. 
By further defining the net load as $D_m = D'_m - p_{solar,m}$, the model can be rewritten as follows.
{
\vspace{-10pt}
\begin{subequations}\label{model:simple_stage}
    \begin{align}
    & \min_{w_i^0,P_i,Q_i,P_{ij},Q_{ij},l_{ij},v_i} \quad \sum_{(i,j)\in \mathcal{E}} r_{ij}l_{ij} \label{obj:USM_new}\\
    & \text{s.t.}  \quad \eqref{cons:node_P}\text{--}\eqref{cons:sum_X}, \notag\\
    & \quad S_i^* = \{p_m^*,p_m^{+*},p_m^{-*},x_m^*\}_{m \in U_i}, \quad \forall i \in \mathcal{N} \notag\\
    & \quad S_i^* = \arg \min_{\substack{p_m,p_m^+,p_m^-,x_m,\\ \forall m \in U_i}} \Big\{ \sum_{m \in U_i} \big(\frac{c_m}{2}p_m^2+b_mp_m \notag\\
    & \qquad \qquad + w^+_ip_m^+ - w_i^-p_m^--w_i^0x_m\big) + \frac{a_i}{2}\sum_{m \in U_i}x_m^2 \notag\\
    & \qquad \qquad + \frac{a_i}{2}\Big(\sum_{m \in U_i}x_m\Big)^2 \Big\} \label{obj:LSM_new}\\
    & \quad \text{s.t.}  \quad \eqref{cons:p_lb_ub}, \eqref{cons:p_+_-}, \notag\\
    & \quad \quad D_m +x_m+p^-_m = p_m + p_m^+, \quad \forall m \in U_i \label{cons:balance_new}
    \end{align}
\end{subequations}
}

In this model, the decision variable $w_i^0$ influences the U-ESM solely through the L-ESM subproblem.
To characterize this influence, we rewrite the L-ESM $i$ subproblem as follows.

\begin{align}\label{model:L-ESM subproblem}
    \min_{\substack{p_m, p_m^+, p_m^-, x_m \\ \forall m \in U_i}} \quad
    & \sum_{m \in U_i} \Big( \frac{c_m}{2}p_m^2 + b_m p_m + w_i^+ p_m^+ - w_i^- p_m^- \\
    \displaybreak[1]
    &- w_i^0 x_m \Big)  + \frac{a_i}{2}\sum_{m \in U_i}x_m^2 + \frac{a_i}{2}\Big(\sum_{m \in U_i}x_m\Big)^2 \nonumber\\
    \text{s.t. } 
    & D_m +x_m+p^-_m = p_m + p_m^+ \,: \lambda_{\text{eq},m}, \forall m \in U_i\nonumber\\
    & 0 \leq p_m \leq \overline{p_m} \,: \lambda_{p,\text{lb},m}, \lambda_{p,\text{ub},m}, \forall m \in U_i\nonumber\\
    & p_m^+ \geq 0,\ p_m^- \geq 0 \,: \lambda_{p,+,m}, \lambda_{p,-,m}, \forall m \in U_i.\nonumber
\end{align}

Specifically, $w_i^0$ affects the U-ESM only through the L-ESM's aggregate variables: the uncleared sharing energy $X_i:= \sum_{m \in U_i} x_m^*$ and
the total power exchange
$P_i:= \sum_{m \in U_i}(x_m^* + p_m^{-*} - p_m^{+*})$.
If we can explicitly characterize the best-response function 
$P_i(X_i)$, the implicit L-ESM equilibrium constraints
can be replaced by explicit functions, transforming
the original bilevel model \eqref{model:simple_stage}
into the following single-level problem:
\begin{subequations}\label{model:simple_stage_with_X}
    \begin{align}
    & \min_{X_i, P_i, Q_i, P_{ij}, Q_{ij}, l_{ij}, v_i} \quad \sum_{(i,j)\in \mathcal{E}} r_{ij}l_{ij} \label{obj:USM_X}\\
    & \text{s.t.} \notag \\
    & \quad P_j = \sum_{k:j \to k} P_{jk} - \sum_{i:i \to j} (P_{ij} - r_{ij}l_{ij}), \forall j \in \mathcal{N} \label{cons:P_node_new}\\
    & \quad P_j = P_j(X_j), \forall j \in \mathcal{N} \label{cons:best_response}\\
    & \quad \eqref{cons:node_Q}, \eqref{cons:node_v}, \eqref{cons:v_lb_ub}, \eqref{cons:l_lb_ub}, 
    \eqref{cons:Q_lb_ub},
    \eqref{cons:sum_X} \notag \\
    & \quad P_{ij}^2 + Q_{ij}^2 \leq l_{ij}v_{i}, \forall (i,j) \in \mathcal{E} \label{cons:line_limit_relaxed}
    \end{align}
\end{subequations}

Thus, clearing the market reduces to deriving
the best-response function  $P_i(X_i)$, which depends on the intrinsic mathematical structure of the subproblem
\eqref{model:L-ESM subproblem}.

\section{Structural Properties of the L-ESM Subproblem}

To derive $P_i(X_i)$, we start with the structural
properties of the L-ESM subproblem
\eqref{model:L-ESM subproblem}. 

\begin{lemma}\label{lem:LSM_KKT}
    For the L-ESM subproblem \eqref{model:L-ESM subproblem},
    let $(p_m^*, p_m^{+*}, p_m^{-*}, x_m^*)$ be an optimal
    solution of prosumer $m$.
     Then one of the following modes must be satisfied:
    \begin{enumerate}
        \item \textbf{Mode 1:} $c_m(x_m^*+D_m)+b_m-w_i^0+a_i(x_m^*+x_{-m}^*)+ a_i x_m^* = 0$ and $p_m^* = x_m^*+D_m$.
        \item \textbf{Mode 2:} $w_i^+-w_i^0 + a_i(x_m^*+x_{-m}^*) + a_i x_m^* = 0$ and $p_m^* = \min(\overline{p_m},\frac{w_i^+-b_m}{c_m})$.
        \item \textbf{Mode 3:} $w_i^--w_i^0 + a_i(x_m^*+x_{-m}^*) + a_i x_m^* = 0$ and $p_m^* = \min(\overline{p_m},\frac{w_i^--b_m}{c_m})$.
        \item \textbf{Mode 4:} $x_m^* = \overline{p_m} - D_m$ and $p_m^* = \overline{p_m}$.
    \end{enumerate}
    Here, $x_{-m}^* := \sum_{j \in U_i, j\neq m}x_j^*$ for all $m \in U_i$, and $U_{i,k} \subseteq U_i$ denotes the set of prosumers in Mode $k$ ($k=1,\dots,4$).
\end{lemma}

The proof of Lemma \ref{lem:LSM_KKT} is provided in Appendix \ref{app:A}.

With the optimal mode of each prosumer identified, we now investigate the functional dependence of the aggregate variables $X_i$ and $P_i$ on the base price $w_i^0$.

Although $X_i$ and $P_i$ are formally defined as variables of the U-ESM, their values are entirely determined by the equilibrium solution of the L-ESM subproblem \eqref{model:L-ESM subproblem} for any given $w_i^0$. Therefore, they can be rigorously treated as functions of $w_i^0$ induced by the L-ESM.

\begin{lemma}\label{lem:piecewise}
    For the L-ESM subproblem \eqref{model:L-ESM subproblem}, the aggregate variables $P_i(w_i^0) = \sum_{m \in U_i}(x_m^*(w_i^0) + p_m^{-*}(w_i^0) - p_m^{+*}(w_i^0))$ and $X_i(w_i^0) = \sum_{m \in U_i} x_m^*(w_i^0)$ are piecewise linear in $w_i^0$. Furthermore, every breakpoint (i.e., a point where the slope changes) of $P_i(w_i^0)$ is necessarily a breakpoint of $X_i(w_i^0)$. 
\end{lemma}

\begin{proof}
    The L-ESM subproblem \eqref{model:L-ESM subproblem} is a strictly convex quadratic program with linear parameter $w_i^0$.
    Standard results from multi-parametric quadratic programming imply that its optimal solution is a continuous piecewise affine function of $w_i^0$ \cite[Theorem~2, Chapter~1]{pistikopoulosMultiparametricProgrammingTheory2007}.
    Because $P_i$ and $X_i$ are linear combinations of that solution, they are also piecewise linear.
    Their breakpoints are induced by changes in the active primal or dual constraint set, i.e., by the mode transitions characterized in Lemma \ref{lem:LSM_KKT}.
\end{proof}

\begin{lemma}\label{lem:dx_dw_i^0}
    In the L-ESM subproblem \eqref{model:L-ESM subproblem}, the shared energy of a specific prosumer $x^*_m$ and the uncleared sharing energy of an L-ESM $X_i$ are non-decreasing with respect to the base price $w_i^0$. 
    Specifically, within each linear segment where the functions are differentiable:
    \begin{equation*}
        \frac{\mathrm{d} x^*_m}{\mathrm{d} w_i^0} \geq 0, \forall m \in U_i;\qquad \frac{\mathrm{d} X_i}{\mathrm{d} w_i^0} \geq 0
    \end{equation*}
    The equality $\frac{\mathrm{d} x^*_m}{\mathrm{d} w_i^0} = 0$ holds if and only if prosumer $m$ is operating in Mode 4. The equality $\frac{\mathrm{d} X_i}{\mathrm{d} w_i^0} = 0$ holds if and only if all prosumers in L-ESM $i$ are in Mode 4.
\end{lemma}

The proof of Lemma \ref{lem:dx_dw_i^0} can be found in Appendix \ref{app:B}.

Note that when $\frac{\mathrm{d} X_i}{\mathrm{d} w_i^0} = 0$
(i.e., all prosumers are in Mode 4), $P_i$  remains
constant with respect to $w_i^0$.
This monotonicity ensures $P_i$ is uniquely determined
for any given $X_i$, guaranteeing the well-posedness of the best-response $P_i(X_i)$.

\begin{remark}
Standard multi-parametric programming typically relies on enumerating active sets to identify critical regions. However, the specific pricing coupling of the L-ESM guarantees the monotonicity of mode transitions (Lemma \ref{lem:dx_dw_i^0}). This key property degenerates the problem into tracing an ordered sequence of breakpoints along a 1D parameter, enabling direct algorithmic construction without region enumeration.
\end{remark}

\section{Best-Response Function of L-ESM and Market-Clearing Algorithm}

Based on these properties, we develop efficient algorithms to determine the best-response function. We first define auxiliary parameters
for each prosumer $m \in U_i$:
\begin{align}
    \alpha_m &:= \frac{w_i^- - b_m}{c_m} - D_m, \\
    \beta_m &:= \frac{w_i^+ - b_m}{c_m} - D_m, \\
    \gamma_m &:= \overline{p_m} - D_m.
\end{align}

Let $n = |U_i|$. The behavior of $X_i(w_i^0)$ near the lower bound of $w_i^0$ is first characterized.

\begin{lemma}\label{lem:initial_segment}
    For the L-ESM subproblem \eqref{model:L-ESM subproblem}, when the base price satisfies $\frac{w_i^0-w_i^-}{a_i(1+n)}\leq \min \{\min_{m \in U_i}(\alpha_m),\min_{m \in U_i}(\gamma_m) \}$, all prosumers operate in Mode 3, yielding $x_m^* = \frac{w_i^0-w_i^-}{a_i(1+n)}$ and $X_i = \frac{w_i^0-w_i^-}{a_i}\frac{n}{1+n}$.
\end{lemma}

\begin{proof}
    According to Lemma \ref{lem:LSM_KKT}, under the given condition, all prosumers operate in Mode 3. Summing the corresponding stationarity conditions yields:
    \begin{equation}
        n(w_i^--w_i^0)+na_iX_i+a_iX_i = 0.
    \end{equation}
    Solving this equation yields the condition for the lemma, and further implies the expression of $x_m^*$.
\end{proof}

As the base price $w_i^0$ increases from the lower bound,
prosumers transition between different modes.
Recall that $U_{i,k} \subseteq U_i$ denotes the set of
prosumers in L-ESM $i$ operating in Mode $k$
(defined in Lemma \ref{lem:LSM_KKT}).
Such mode transitions introduce breakpoints to
the piecewise linear function $X_i(w_i^0)$, as characterized by the following proposition.
\begin{proposition}\label{pro:prosumer_move}
    In the L-ESM subproblem \eqref{model:L-ESM subproblem},
    for any prosumer $m$, the migration path with increasing $w_i^0$ is determined by the relationship between the  parameters $\alpha_m, \beta_m, \gamma_m$:
    \begin{enumerate}
        \item  $\gamma_m > \beta_m$: the path is Mode 3 $\to$ Mode 1 $\to$ Mode 2.
        \item $\gamma_m \leq \alpha_m$: the path is Mode 3 $\to$ Mode 4 $\to$ Mode 2.
        \item $\alpha_m < \gamma_m \leq \beta_m$: the path is Mode 3 $\to$ Mode 1 $\to$ Mode 4 $\to$ Mode 2.
    \end{enumerate}
    Notably, if all prosumers enter Mode 4, $X_i$ becomes constant with respect to $w_i^0$ and all prosumers remain locked in Mode 4 regardless of further price increases. The specific conditions triggering transitions are:
    \begin{enumerate}
        \item $U_{i,3}\to U_{i,1}$: $p_m^*-D_m = x_m^* = \alpha_m$.
        \item $U_{i,3}\to U_{i,4}$: $p_m^* = \overline{p_m} = x_m^*+D_m$ ($x_m^* = \gamma_m$).
        \item $U_{i,1}\to U_{i,4}$: $p_m^* = \overline{p_m} = x_m^*+D_m$ ($x_m^* = \gamma_m$).
        \item $U_{i,1}\to U_{i,2}$: $p_m^*-D_m = x_m^* = \beta_m$.
        \item $U_{i,4}\to U_{i,2}$: $X_i = \frac{w_i^0-w_i^+}{a_i}-\gamma_m$.
    \end{enumerate}
\end{proposition}

The proof of Proposition \ref{pro:prosumer_move} can be found in Appendix \ref{app:D}.

Based on the piecewise linear structure and the monotonic mode transitions characterized in Proposition \ref{pro:prosumer_move}, the functions $X_i(w_i^0)$ and $P_i(X_i)$ can be efficiently constructed for arbitrary prosumer parameters by iteratively identifying the mode transitions and updating the linear coefficients.

We first present the algorithm for computing $X_i(w_i^0)$. 
The core idea is to initialize at the lowest price segment 
where all prosumers operate in Mode 3 and iteratively update their states and linear coefficients as the price increases.

\begin{algorithm}
    \caption{General Computation of $X_i(w_i^0)$}
    \label{alg:X_general}
    \KwIn{Main grid buying/selling prices $w_i^+, w_i^-$; 
          Price elasticity $a_i$; 
          Prosumer parameters $\{c_m, b_m, D_m, \overline{p_m}\}_{m \in U_i}$}
    \KwOut{Coefficients $\{k'_{X,j}, b'_{X,j}\}$ and  breakpoints $BP_j$ 
           for each linear segment}

    \tcc{Initialization}
    Calculate auxiliary parameters $\alpha_m, \beta_m, \gamma_m$ 
    for all $m \in U_i$\;
    Initialize the state of each prosumer $m$, denoted as $\text{State}_m$, 
    to Mode 3\;
    Set segment index $j \leftarrow 0$ and breakpoint $BP_0 \leftarrow -\infty$\;

    \While{true}{
        \tcc{Step 1: Construct linear segment $j$}
        Sum the stationarity equations based on current states $\text{State}_m$ 
        (Lemma \ref{lem:LSM_KKT}) to derive the affine function 
        $X_i = k'_{X,j} w_i^0 + b'_{X,j}$\;
        
        \tcc{Step 2: Determine next breakpoint}
        \For{each prosumer $m \in U_i$}{
            Calculate the candidate trigger price $w_{m}^{next}$ 
            required to transition from $\text{State}_m$ 
            to the next mode based on Proposition \ref{pro:prosumer_move}\;
        }
        $BP_{j+1} \leftarrow \min_{m} \{w_{m}^{next} \mid w_{m}^{next} > BP_j\}$\;
        
        \If{$BP_{j+1} = +\infty$}{
            Break the loop\;
        }
        
        \tcc{Step 3: Update state}
        Identify the prosumer $m^*$ corresponding to $BP_{j+1}$ 
        and update $\text{State}_{m^*}$ to the next mode\;
        $j \leftarrow j + 1$\;
    }
    \Return{The set of coefficients and breakpoints for all segments}
\end{algorithm}

With the piecewise linear function $X_i(w_i^0)$ and the prosumer 
states determined by Algorithm \ref{alg:X_general}, we now proceed 
to construct the best-response function  $P_i(X_i)$. 
Algorithm \ref{alg:P_general} presents the detailed procedure 
for this derivation.

\begin{algorithm}
    \caption{General Computation of $P_i(X_i)$}
    \label{alg:P_general}
    \KwIn{State sequence $\{\text{State}_m^{(j)}\}$ and coefficients 
          $\{k'_{X,j}, b'_{X,j}\}$ from Algorithm \ref{alg:X_general}; 
          Elasticity coefficient $a_i$; 
          Prosumer parameters $\{c_m, b_m, D_m\}_{m \in U_i}$; 
          Auxiliary parameters $\{\alpha_m, \beta_m, \gamma_m\}_{m \in U_i}$}
    \KwOut{Coefficients $\{k_{P,j}, b_{P,j}\}$ for the piecewise 
           linear function $P_i = k_{P,j}X_i + b_{P,j}$}

    \tcc{Iterate through each linear segment}
    \For{each segment index $j$}{
        \tcc{Step 1: Obtain inverse coefficients for substitution}
        Initialize $k_{P,j} \leftarrow 0$, $b_{P,j} \leftarrow 0$\;
        \eIf{$k'_{X,j} \neq 0$}{
            $k_{X,j} \leftarrow 1/k'_{X,j}$\;
            $b_{X,j} \leftarrow -b'_{X,j}/k'_{X,j}$\;
        }{
            Treat segment $j$ as a constant-$X_i$ segment and compute $P_i$ directly from the current prosumer states\;
        }
        
        \tcc{Step 2: Calculate $P_i(X_i)$ coefficients}
        \For{each prosumer $m \in U_i$}{
            \If{$\text{State}_m^{(j)} = \text{Mode 1}$}{
                $k_{P,j} \leftarrow k_{P,j} + \frac{k_{X,j} - a_i}{c_m + a_i}$\;
                $b_{P,j} \leftarrow b_{P,j} + \frac{b_{X,j} - b_m - c_m D_m}{c_m + a_i}$\;
            }
            \ElseIf{$\text{State}_m^{(j)} = \text{Mode 2}$}{
                $b_{P,j} \leftarrow b_{P,j} + \min(\beta_m, \gamma_m)$\;
            }
            \ElseIf{$\text{State}_m^{(j)} = \text{Mode 3}$}{
                $b_{P,j} \leftarrow b_{P,j} + \min(\alpha_m, \gamma_m)$\;
            }
            \ElseIf{$\text{State}_m^{(j)} = \text{Mode 4}$}{
                $b_{P,j} \leftarrow b_{P,j} + \gamma_m$\;
            }
        }
    }
    \Return{$\{k_{P,j}, b_{P,j}\}$}
\end{algorithm}

Model \eqref{model:simple_stage_with_X} becomes an MISOCP.
To ensure tractability, we establish the tightness of
its SOCP relaxation.

\begin{assumption}\label{ass:reactive_support}
For each node $j \in \mathcal{N}$, the reactive-power support variable $Q_j$ is adjustable within the compact interval $[\underline{Q_j},\overline{Q_j}]$. 
\end{assumption}

\begin{proposition}\label{pro:SOCP_relax}
    Under Assumption \ref{ass:reactive_support}, for the single-level optimization problem
    \eqref{model:simple_stage_with_X}, 
    the convex relaxation of the branch power flow $P_{ij}^2 + Q_{ij}^2 \leq l_{ij}v_{i}, \forall (i,j) \in \mathcal{E}$ is tight.
\end{proposition}

The proof of Proposition \ref{pro:SOCP_relax} can be found in Appendix \ref{app:F}.

The overall market-clearing procedure is summarized in
Algorithm \ref{alg:overall_clearing}.

\begin{algorithm}
\caption{Non-Iterative Clearing for Two-Layer Energy-Sharing Market} \label{alg:overall_clearing}
\KwIn{Network parameters; L-ESM parameters $\{w_i^+, w_i^-, a_i, \{c_m, b_m, D_m, \overline{p_m}\}_{m \in U_i}\}_{i \in \mathcal{N}}$}
\KwOut{Optimal base prices $w_i^{0*}$, sharing prices $w_i^*$, uncleared sharing energy $X_i^*$, and total power exchange $P_i^*$}
\For{each L-ESM $i \in \mathcal{N}$}{
Run Algorithm \ref{alg:X_general} to obtain segments and breakpoints of $X_i(w_i^0)$\;
Run Algorithm \ref{alg:P_general} to construct the piecewise linear function $P_i(X_i)$\;
}
Formulate the single-level optimization problem \eqref{model:simple_stage_with_X} using the obtained piecewise linear functions $P_i(X_i)$\;
Solve the MISOCP to obtain the optimal upper-layer solution, including $X_i^*$ and $P_i^*$\;
Determine the operating segment for each L-ESM based on $X_i^*$ and calculate the corresponding base price $w_i^{0*}$\;
Calculate the local sharing price $w_i^* = w_i^{0*} - a_i X_i^*$\;
\tcc{Verification Step}
\For{each L-ESM $i \in \mathcal{N}$}{
Solve the L-ESM subproblem \eqref{model:L-ESM subproblem} with the obtained $w_i^{0*}$\;
Verify that the resulting optimal aggregate variables match $X_i^*$ and $P_i^*$\;
}
\Return{$w_i^{0*}, w_i^*, X_i^*, P_i^*$}
\end{algorithm}

\section{Case Study}

\subsection{Setup}

\begin{figure}[t]
    \centering
    \includegraphics[width=\linewidth]{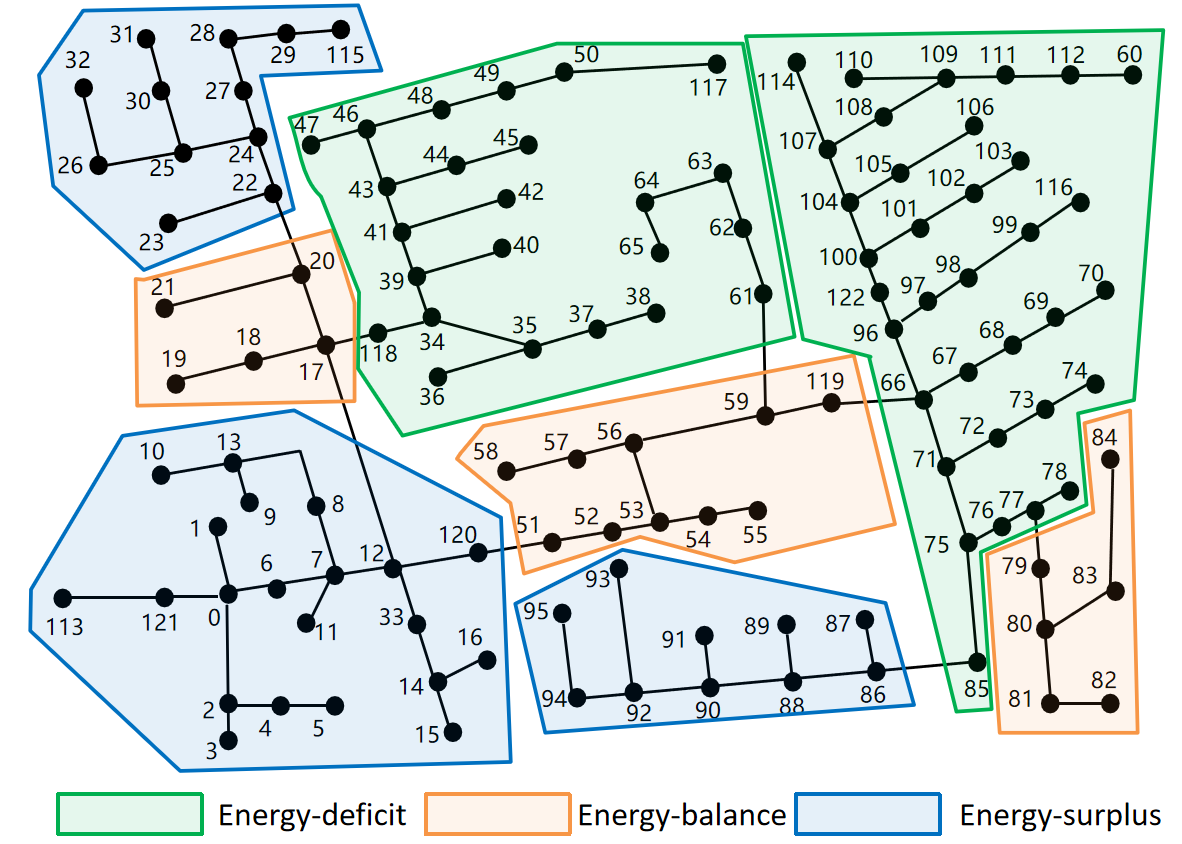}
    \caption{Configuration of the IEEE 123-bus test system with regional
             supply-demand scenarios.}
    \label{fig:case_topology}
\end{figure}

In this section, simulation experiments are conducted using Python and Gurobi on a desktop with an Intel Core i7-14650HX CPU.
The IEEE 123-bus distribution feeder is used as the test system, with network parameters adopted from the simplified MATPOWER model in \cite{boboSecondOrderConeRelaxations2020}. Each bus is associated with one L-ESM, and each L-ESM contains 100 prosumers, resulting in a total of 12,300 prosumers. The study considers single-period market clearing over a 1-h interval. Under this setting, active power values in kW are numerically equivalent to energy quantities in kWh. Prosumer parameters are synthetically generated, and the same random realization is shared across all compared methods to ensure a fair comparison.

Since renewable energy generation is fluctuating, we manually partition the
system into three typical scenarios, as illustrated in Fig. \ref{fig:case_topology}:
\begin{itemize}
    \item \textbf{Energy-surplus region:} L-ESMs in this area are set with
          abundant non-adjustable generation. The net load $D$ for prosumers
          is randomly selected from $[-40, -20]$ kW.
    \item \textbf{Energy-deficit region:} L-ESMs in this area are set with
          heavy loads and minimal local generation. The net load $D$ for
          prosumers is randomly selected from $[20, 40]$ kW.
    \item \textbf{Energy-balance region:} L-ESMs in this area maintain a
          relative balance. The net load $D$ for prosumers is randomly
          selected from $[-5, 5]$ kW.
\end{itemize}

Following \cite{suSharingEnergyWider2025}, the buying and selling prices
$w_i^+, w_i^-$ of the electric utility are set to $0.2\$/\text{kW}$ and
$0.05\$/\text{kW}$, respectively.
For each L-ESM, the price elasticity $a_i$ is randomly chosen within the
range $[2.5, 5]\times 10^{-3}/|U_i| \$/\text{kW}$.
The parameters of prosumers are randomly chosen within the ranges:
$c_m \in [0.5, 1]\times 10^{-3}\$/\text{kW}^2$,
$b_m \in [0.01, 0.05]\$/\text{kW}$, and
$\overline{p_m} \in [0, 40]$ kW.

\subsection{Validation of Proposed Algorithmic Solution}

To verify the validity of the proposed algorithmic solution, four L-ESMs are tested. The first L-ESM contains 30 prosumers, while the remaining three L-ESMs each consist of 100 prosumers. For the 100-prosumer L-ESMs, the price elasticity coefficients are set to $a = 3 \times 10^{-5}$, $a = 4 \times 10^{-5}$, and $a = 5 \times 10^{-5}$, respectively. Other parameters are randomized as described in Section V.A.

The comparison involves two solution types: numerical solutions from direct optimization and algorithmic solutions. Specifically, independent variables ($w_i^0$ or $X_i$) at each breakpoint of the algorithmic solution are substituted into the optimization problem to evaluate and compare dependent variables ($X_i$ or $P_i$).

Fig. \ref{fig:X_P_error} shows the results. The top-left and top-right subfigures compare $X_i(w_i^0)$ and $P_i(X_i)$, respectively, with all data points aligning precisely on the diagonal, confirming excellent agreement between algorithmic and numerical solutions. The bottom-left and bottom-right subfigures display relative errors for $X_i$ and $P_i$, with maximum errors of approximately 0.007\% and 0.005\%, respectively. Most errors are significantly smaller, validating the high accuracy of the algorithmic solution.

\begin{figure}
    \centering
    \includegraphics[width=\linewidth]{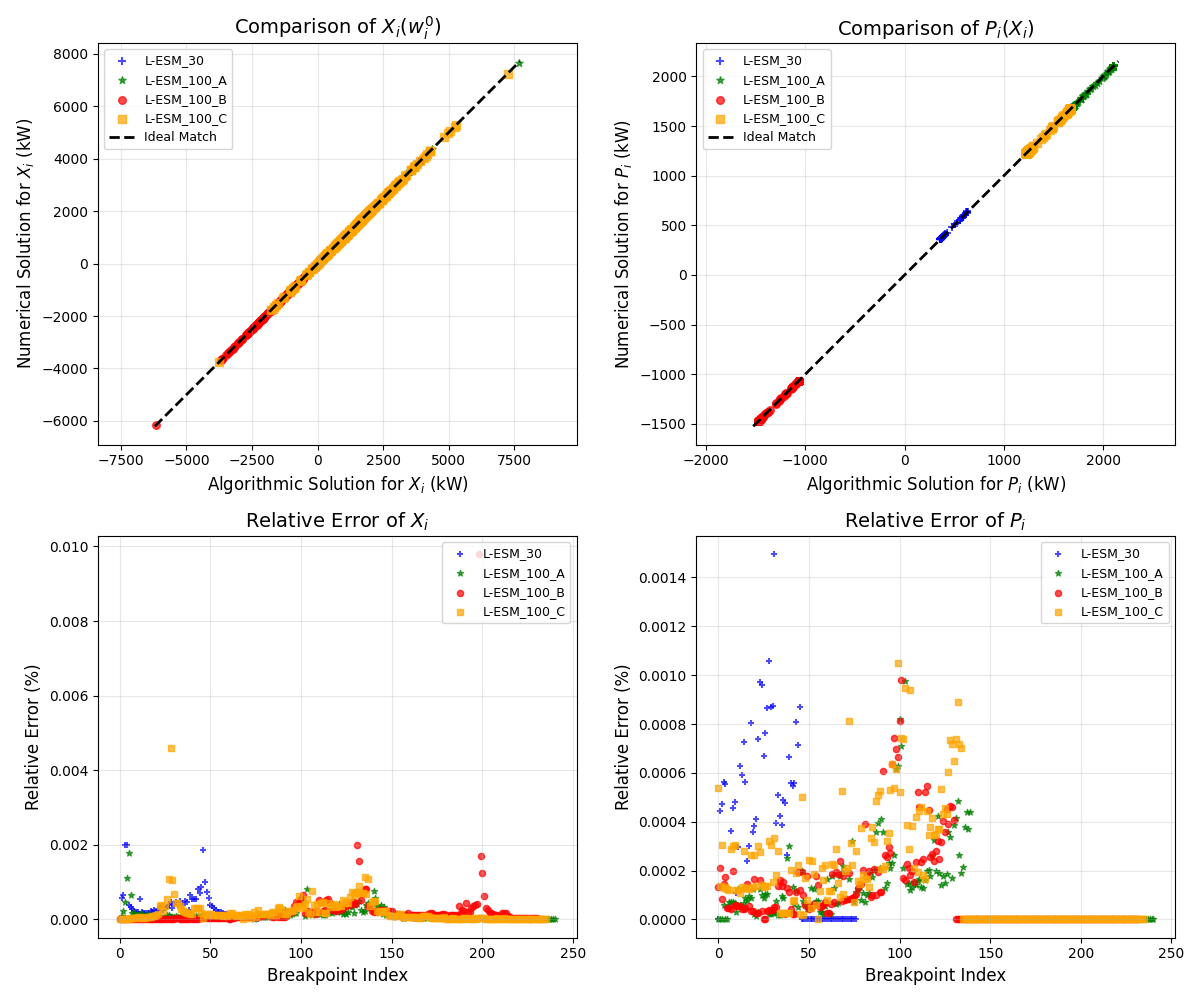}
    \caption{Algorithmic vs. numerical solutions for L-ESMs with varying prosumer counts and parameters. Top: $X_i(w_i^0)$ and $P_i(X_i)$ comparisons. Bottom: Relative errors of $X_i$ and $P_i$.}
    \label{fig:X_P_error}
\end{figure}

\subsection{Market Clearing Results and Voltage Validation}
This subsection presents the market-clearing results for the two-layer energy-sharing market and validates the physical feasibility of the obtained solution, particularly the satisfaction of voltage security constraints.

Fig. \ref{fig:market_clearing} illustrates the spatial distribution of the optimal clearing outcomes across the network. The base prices $w_i^0$ determined by the U-ESM exhibit clear regional variation. Nodes in the energy-surplus region tend to have higher base prices, reaching approximately 0.23 \$/kWh, which encourages energy export toward regions with stronger demand. By contrast, nodes in the energy-deficit region exhibit lower base prices, around -0.10 \$/kWh, to encourage energy import. Compared with the base-price profile, the corresponding sharing prices, $w_i = w_i^0 - a_i X_i$, are noticeably smoother across the feeder, suggesting that the local sharing mechanism mitigates regional price disparities while coordinating supply and demand.

\begin{figure}[h]
\centering
\includegraphics[width=1\linewidth]{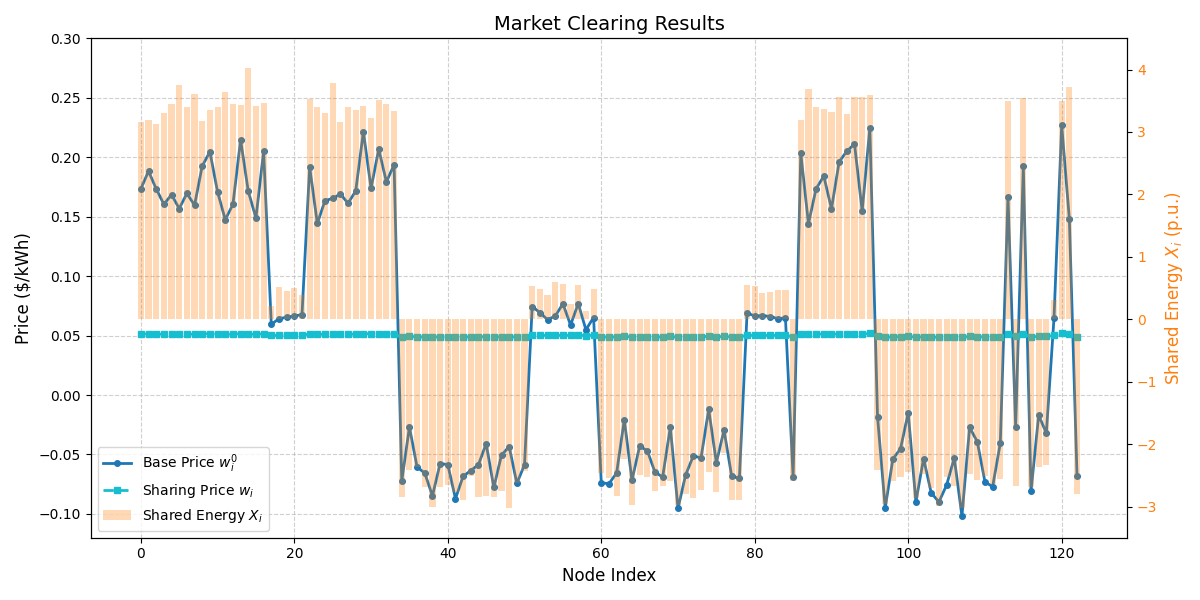}
\caption{Spatial distribution of optimal clearing results: (a) base prices $w_i^0$, (b) sharing prices $w_i$, and (c) shared energy $X_i$.}
\label{fig:market_clearing}
\end{figure}

Physical feasibility requires maintaining node voltages within secure limits. Fig. \ref{fig:voltage_comparison} and Fig. \ref{fig:PQ_comparison} compare the voltage profiles and nodal active/reactive power injection obtained by the proposed method against the benchmark method by Su et al. \cite{suSharingEnergyWider2025}, which employs a linearized power flow model and neglects explicit voltage constraints.

\begin{figure}[h]
\centering
\includegraphics[width=1\linewidth]{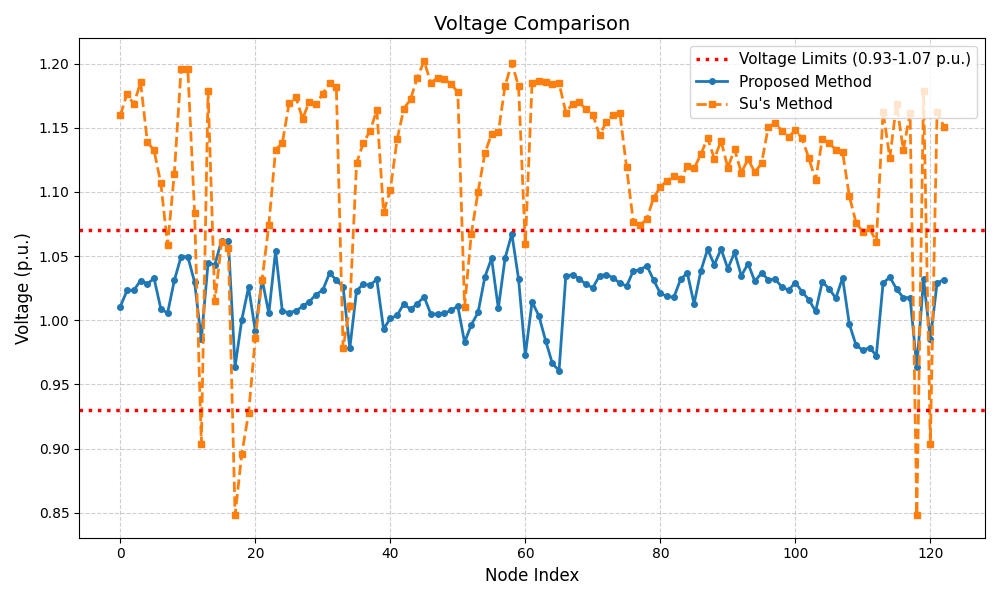}
\caption{Comparison of voltage profiles: Proposed method vs. Su's method \cite{suSharingEnergyWider2025}. The proposed method successfully maintains all voltages within the secure range [0.93, 1.07] p.u. (marked by red dashed lines), while the benchmark method suffers severe overvoltage violations.}
\label{fig:voltage_comparison}
\end{figure}

\begin{figure}
\centering
\includegraphics[width=1\linewidth]{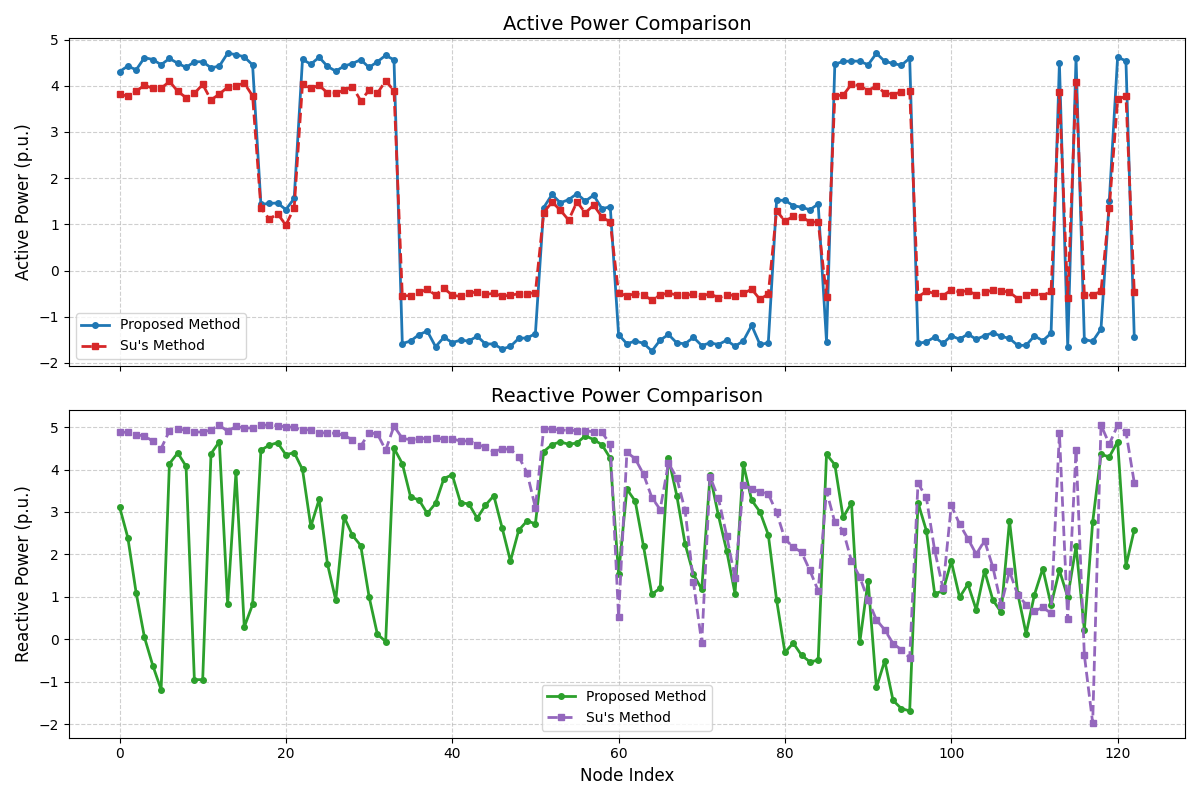}
\caption{Comparison of nodal active and reactive power injection profiles.}
\label{fig:PQ_comparison}
\end{figure}

The proposed algorithm maintains all node voltages within the secure range (0.93–1.07 p.u.). In contrast, the benchmark method causes severe violations, with voltages exceeding 1.20 p.u. at multiple nodes. These results confirm that rigorous power flow constraints are essential to ensuring the physical feasibility of market-clearing outcomes. The proposed method effectively bridges the gap between market efficiency and physical security by integrating voltage constraints directly into the non-iterative clearing algorithm.

\subsection{Computation Time Comparison}

In this subsection, we compare the computation time of the following four methods:

\begin{enumerate}
    \item \textbf{The Proposed Method:} The method introduced in this paper.
    \item \textbf{Su's Method \cite{suSharingEnergyWider2025}:} A benchmark method that does not consider voltage constraints.
    \item \textbf{The Centralized Method:} An approach that involves solving the entire optimization problem directly.
    \item \textbf{The ADMM Method:} A bilevel iterative method that incorporates both voltage constraints and the individual interests of prosumers.
\end{enumerate}

To ensure a fair comparison reflecting practical implementation, the computation times reported in Table \ref{tab:time_comparison} account for parallel processing where applicable: Su's method considers prosumer-level parallelism; the proposed method considers L-ESM-level parallelism; the ADMM method considers parallel computation of the lower-layer problems; whereas the centralized method cannot be parallelized.

\begin{table}[h!]
    \centering
    \caption{Comparison of computation times for different numbers of prosumers.}
    \label{tab:time_comparison}
    \begin{tabular}{l c c c c c}
        \toprule
        Prosumer Count & 123 & 369 & 1,230&  6,150&12,300\\
        \midrule
        {Proposed Method} & 0.33ms & 0.48ms & 1.9ms & 0.577s & 0.829s\\
        {Su's Method} & 4.03ms & 40.6ms & 0.12s & 0.028s &  0.473s\\
        {Centralized Method} & 0.70s & 148.06s & $\geq$ 1h& - &-\\
        {ADMM Method}& 43.28s & 149.02s & $\geq$ 1h & - &-\\
        \bottomrule
    \end{tabular}
\end{table}

Table \ref{tab:time_comparison} shows that the proposed method scales favorably as the system size increases. The centralized formulation becomes computationally prohibitive once the number of prosumers reaches $10^3$, and the ADMM-based method is likewise hindered by slow convergence under voltage-coupled constraints. Su's method, relying on a linearized power-flow model, remains computationally efficient, but it does not provide explicit voltage-security guarantees. By contrast, the proposed method solves the 12,300-prosumer case in 0.829~s while maintaining all nodal voltages within their limits, supporting its use for large-scale near-real-time market clearing.

\subsection{Social Efficiency}

To illustrate the economic necessity of a hierarchical, distributed market for energy sharing, we compare the average costs of prosumers across four market configurations: No Sharing (NS), Local Sharing (LS), Global Sharing (GS), and Global Sharing without Voltage Constraints (GS-NVC).

\begin{enumerate}
\item \textbf{No Sharing (NS):} Prosumers balance independently
($x_m = 0,\forall m \in U_i, \forall i \in \mathcal{N}$).
\item \textbf{Local Sharing (LS):} Prosumers trade only locally
($X_i = 0,\forall i \in \mathcal{N}$).
\item \textbf{Global Sharing (GS):} The proposed method.
\item \textbf{GS without Voltage Constraints (GS-NVC):} The proposed
method without enforcing voltage constraints, representing the
economic efficiency when physical security is relaxed.
\end{enumerate}

\begin{table}[h!]
    \centering
    \caption{Comparison of cost for different energy sharing situations.}
    \label{tab:cost_comparison}
    \begin{tabular}{l c c c c}
        \toprule
         Situations & NS & LS & GS & GS-NVC\\
        \midrule
        Cost $\$/\text{kWh}$&0.03551&0.03541&0.03454&0.03452\\
        Total cost $\$$&6,550.84&6,532.40&6,371.90&6,368.21\\
        \bottomrule
    \end{tabular}
\end{table}
Table \ref{tab:cost_comparison} shows a monotonic decrease in average cost as the sharing scope expands. Moving from NS to LS yields only a limited reduction, from 0.03551 to 0.03541 \$/\text{kWh}, because the restriction $X_i=0$ confines balancing to individual L-ESMs and does not exploit inter-regional complementarities. The GS configuration reduces the average cost further to 0.03454 \$/\text{kWh} by allowing the U-ESM to coordinate energy exchange across the feeder. The GS-NVC case yields a slightly lower cost of 0.03452 \$/\text{kWh} but exhibits serious voltage limit violation, which illustrates the expected trade-off between economic performance and network security. Overall, the results suggest that the proposed hierarchical design improves economic efficiency while preserving voltage feasibility.

\section{Conclusion}
This paper proposes a non-iterative clearing algorithm for the two-layer energy-sharing market with voltage constraints. By deriving the best-response function of the lower-layer market, the original bilevel MPEC is reformulated as a single-level MISOCP that remains tractable for large-scale systems.

Case studies of the IEEE 123-bus system with up to 12,300 prosumers yield three main findings. First, the derived best-response function accurately reproduces the lower-layer equilibrium solution. Second, the proposed method maintains nodal voltages within prescribed limits, unlike benchmark approaches that do not explicitly enforce them. Third, the resulting clearing problem can be solved within sub-second time for the largest test case considered, which suggests strong potential for near-real-time implementation. 

The present study is limited to a deterministic, single-period setting and a radial distribution network model. Future work will extend the framework to multi-period coordination with energy storage, uncertainty-aware market clearing, and privacy-preserving implementations.

\begin{appendices}
    \section{Proof of Lemma \ref{lem:LSM_KKT}}
    \label{app:A}
    \begin{proof}
    Applying the KKT conditions yields the following stationarity conditions:
    \begin{subequations}\label{app:KKT}
        \begin{align}
            c_m p_m + b_m - \lambda_{\text{eq},m} - \lambda_{p,\text{lb},m} + \lambda_{p,\text{ub},m} &= 0 \label{app:KKT_p}\\
            w_i^+ - \lambda_{\text{eq},m} - \lambda_{p,+,m} &= 0 \label{app:KKT_p+}\\
            -w_i^- + \lambda_{\text{eq},m} - \lambda_{p,-,m} &= 0 \label{app:KKT_p-}\\
            -w_i^0 + a_i x_m + a_i(x_m + x_{-m}) + \lambda_{\text{eq},m} &= 0 \label{app:KKT_x}
        \end{align}
    \end{subequations}
    
    According to Proposition 1 in \cite{suSharingEnergyWider2025}, the optimal solution automatically satisfies the complementarity constraint $p_m^+ p_m^- = 0$. 
    We next analyze the following three mutually exclusive cases for the trading variables $p_m^+$ and $p_m^-$:
    
    \textbf{Case 1:} $p_m^+ = 0$ and $p_m^- = 0$.
    From the power balance constraint, we have $p_m = D_m + x_m$.
    Substituting $\lambda_{\text{eq},m}$ from \eqref{app:KKT_x} into \eqref{app:KKT_p} and considering the bounds of $p_m$:
    \begin{itemize}
        \item If $p_m < \overline{p_m}$, then $\lambda_{p,\text{ub},m} = 0$. Rearranging the equations yields the condition for \textbf{Mode 1}.
        \item If $p_m = \overline{p_m}$, it implies the adjustable generator reaches its upper limit, corresponding to \textbf{Mode 4}.
    \end{itemize}

    \textbf{Case 2:} $p_m^+ > 0$ and $p_m^- = 0$.
    This implies $\lambda_{p,+,m} = 0$. From \eqref{app:KKT_p+}, we have $\lambda_{\text{eq},m} = w_i^+$. Substituting this into \eqref{app:KKT_x} yields the condition for \textbf{Mode 2}.

    \textbf{Case 3:} $p_m^- > 0$ and $p_m^+ = 0$.
    This implies $\lambda_{p,-,m} = 0$. From \eqref{app:KKT_p-}, we have $\lambda_{\text{eq},m} = w_i^-$. Substituting this into \eqref{app:KKT_x} yields the condition for \textbf{Mode 3}.

    \end{proof}

    \section{Proof of Lemma \ref{lem:dx_dw_i^0}} 
    \label{app:B}
        \begin{proof}
    By Lemma \ref{lem:piecewise}, the optimal solution is piecewise linear in $w_i^0$.
    Therefore, the derivatives exist within each linear segment. We analyze the derivatives on a segment over which the active constraint set remains unchanged.
    
    Based on the modes defined in Lemma \ref{lem:LSM_KKT}, we partition the prosumers into three sets for differentiation:
    \begin{itemize}
        \item $\Omega_1$: Prosumers in Mode 1.
        \item $\Omega_{23}$: Prosumers in Mode 2 or Mode 3.
        \item $\Omega_4$: Prosumers in Mode 4.
    \end{itemize}
    This partition reflects the distinct forms of the associated differential equations.
    
    Differentiating the equations satisfied by the prosumer components of the optimal solution in Lemma \ref{lem:LSM_KKT} yields:
    
    \textbf{Case 1 ($m \in \Omega_1$):} $p_m^* = x_m^* + D_m$. From the stationarity condition:
$
        (c_m + a_i) \frac{\mathrm{d} x^*_m}{\mathrm{d} w_i^0} + a_i \frac{\mathrm{d} X_i}{\mathrm{d} w_i^0} = 1
$

    \textbf{Case 2 ($m \in \Omega_{23}$):} $p_m^+$ or $p_m^-$ is active. The stationarity conditions for these modes yield the same differential equation structure:
$
        a_i \frac{\mathrm{d} x^*_m}{\mathrm{d} w_i^0} + a_i \frac{\mathrm{d} X_i}{\mathrm{d} w_i^0} = 1
$
    
    \textbf{Case 3 ($m \in \Omega_4$):} $p_m^* = \overline{p_m}$ and $x_m^* = \overline{p_m} - D_m$. This implies $x_m^*$ is constant:
$
        \frac{\mathrm{d} x^*_m}{\mathrm{d} w_i^0} = 0
$
    Denoting $x'_m := \frac{\mathrm{d} x^*_m}{\mathrm{d} w_i^0}$ and $X'_i := \frac{\mathrm{d} X_i}{\mathrm{d} w_i^0}$, we can rearrange the equations for $\Omega_1$ and $\Omega_{23}$ as:
$
        x'_m = \frac{1 - a_i X'_i}{c_m + a_i}, \quad \forall m \in \Omega_1 
$ and
$
        x'_m = \frac{1 - a_i X'_i}{a_i}, \quad \forall m \in \Omega_{23}
$
    
    Summing the individual derivatives gives the derivative of the aggregate variable $X'_i = \sum_{m \in U_i} x'_m$:
    \begin{equation*}
        \sum_{m \in \Omega_1} \frac{1 - a_i X'_i}{c_m + a_i} + \sum_{m \in \Omega_{23}} \frac{1 - a_i X'_i}{a_i} + \sum_{m \in \Omega_4} 0 = X'_i
    \end{equation*}
    
    Rearranging the terms gives the equation for $X'_i$:
    \begin{equation*}
        \left( 1 + \sum_{m \in \Omega_1} \frac{a_i}{c_m + a_i} + |\Omega_{23}| \right) X'_i = \sum_{m \in \Omega_1} \frac{1}{c_m + a_i} + \frac{|\Omega_{23}|}{a_i}
    \end{equation*}
    
    We now distinguish the result according to the composition of these sets:
    \begin{enumerate}
        \item If $\Omega_1 \cup \Omega_{23} \neq \emptyset$, the right-hand side is strictly positive, and the coefficient on the left-hand side is positive. Thus, $X'_i > 0$.
        \item If $\Omega_1 \cup \Omega_{23} = \emptyset$ (i.e., all prosumers are in $\Omega_4$), the equation reduces to $X'_i = 0$.
    \end{enumerate}

    Consequently, $X'_i \geq 0$ always holds. Substituting $X'_i \geq 0$ back into the individual equations:
    \begin{itemize}
        \item For $m \in \Omega_1 \cup \Omega_{23}$: Since $X'_i \leq \frac{1}{a_i}$ (implied by the equation structure), we have $x'_m \geq 0$. Specifically, if $X'_i > 0$, then $x'_m > 0$.
        \item For $m \in \Omega_4$: $x'_m = 0$ by definition.
    \end{itemize}
    
    This result establishes both the monotonicity and equality conditions, thereby completing the proof.
    \end{proof}

    \section{Proof of Proposition \ref{pro:prosumer_move}}
    \label{app:D}
    \begin{proof}
    By Lemma \ref{lem:LSM_KKT}, the boundaries between modes correspond to specific values of the shared energy $x_m^*$:
    \begin{itemize}
        \item $U_{i,3} \to U_{i,1}$ occurs when $x_m^* = \frac{w_i^--b_m}{c_m} - D_m = \alpha_m$.
        \item $U_{i,1} \to U_{i,2}$ occurs when $x_m^* = \frac{w_i^+-b_m}{c_m} - D_m = \beta_m$.
        \item Entry into Mode 4 occurs when $x_m^* = \overline{p_m} - D_m = \gamma_m$.
    \end{itemize}

    Since $x_m^*$ is non-decreasing with respect to $w_i^0$ (Lemma \ref{lem:dx_dw_i^0}), the sequence of mode transitions strictly follows the ascending order of the set $\{\alpha_m, \beta_m, \gamma_m\}$:
    \begin{enumerate}
        \item If $\gamma_m > \beta_m$, the order is $\alpha_m < \beta_m < \gamma_m$, yielding the path Mode 3 $\to$ Mode 1 $\to$ Mode 2.
        \item If $\gamma_m \leq \alpha_m$, the order is $\gamma_m \leq \alpha_m < \beta_m$, yielding the path Mode 3 $\to$ Mode 4 $\to$ Mode 2.
        \item Otherwise, the order is $\alpha_m < \gamma_m \leq \beta_m$, yielding the path Mode 3 $\to$ Mode 1 $\to$ Mode 4 $\to$ Mode 2.
    \end{enumerate}

    Finally, consider the locking mechanism. If $U_i = U_{i,4}$, then $X_i = \sum_{m \in U_i} \gamma_m$ is constant. Because the optimal state $x_m^* = \gamma_m$ is enforced by the physical capacity limit $\overline{p_m}$, the partial derivative $\frac{\partial x_m^*}{\partial w_i^0} = 0$. Moreover, the KKT conditions remain satisfied for any $w_i^0$ large enough to maintain saturation. Hence, all prosumers remain locked in Mode 4.
\end{proof}

        \section{Proof of Proposition \ref{pro:SOCP_relax}}
    \label{app:F}
\begin{proof}
Let $\mathcal{S}^* = \{X_i^*, P_j^*, Q_j^*, P_{ij}^*, Q_{ij}^*, l_{ij}^*, v_i^*\}$ denote an optimal solution set. 
Given the optimal uncleared energy $X_j^*$ from the lower-level equilibrium, the active power injection $P_j^* = P_j(X_j^*)$ can be treated as a fixed parameter in the upper-level constraints. Under Assumption \ref{ass:reactive_support}, tightness is proved by backward induction on the radial network topology.

\textbf{1. Leaf Branches:}
Consider a branch $(i,j) \in \mathcal{E}$ where $j$ is a leaf node. At the optimal solution, the power balance equations \eqref{cons:node_P} and \eqref{cons:node_Q} become $
    P_{ij}^* = r_{ij}l_{ij}^* - P_j^*$ and $
    Q_{ij}^* = x_{ij}l_{ij}^* - Q_j^*$. 
We consider two cases up to the status of voltage constraints.

\textbf{Case 1: $v_j^* < \overline{v_j}$.}
Assume for contradiction that the relaxation is not tight, i.e., $l_{ij}^* v_i^* > P_{ij}^{*2} + Q_{ij}^{*2}$.
Substituting the optimal values $P_{ij}^*$ and $P_{ij}^*$ into this inequality yields:
\begin{equation}\label{eq:app_region}
    l_{ij}^* v_i^* - (r_{ij}l_{ij}^* - P_j^*)^2 - (x_{ij}l_{ij}^* - Q_j^*)^2 > 0.
\end{equation}
Consider the function $f(l) = l v_i^* - (r_{ij}l - P_j^*)^2 - (x_{ij}l - Q_j^*)^2$. This function is a strictly concave quadratic in $l$. The feasible region for $l_{ij}$ is the closed interval $\mathcal{L}_{ij} = \{l \mid f(l) \geq 0\}$.
Since $f(l_{ij}^*) > 0$, the point $l_{ij}^*$ lies in the interior of $\mathcal{L}_{ij}$. Therefore, there exists some $l'_{ij} < l_{ij}^*$ such that $f(l'_{ij}) \geq 0$. Because $v_j^* < \overline{v_j}$, $l'_{ij}$ can be chosen sufficiently close to $l_{ij}^*$ so that the voltage constraint $v_j(l'_{ij}) \leq \overline{v_j}$ remains satisfied. This feasible point yields a smaller objective value, which contradicts the optimality of $l_{ij}^*$. Hence, the constraint must be tight.

\textbf{Case 2: $v_j^* = \overline{v_j}$.}
In this scenario, decreasing $l_{ij}$ increases $v_j$, potentially violating the upper voltage limit. To analyze the trade-off, we express the voltage magnitude $v_j$ as an affine function of the decision variables $l_{ij}$ and $Q_j$:
\begin{equation}
    v_j(l_{ij}, Q_j) = v_i^* - z_{ij}l_{ij} + 2r_{ij}P_j^* + 2x_{ij}Q_j,
\end{equation}
where $z_{ij} = r_{ij}^2 + x_{ij}^2$. The partial derivatives  are: $   \frac{\partial v_j}{\partial l_{ij}} = -z_{ij} < 0$, and $\frac{\partial v_j}{\partial Q_j} = 2x_{ij} > 0$, respectively.

Thus, decreasing $l_{ij}$ increases $v_j$, whereas decreasing $Q_j$ decreases $v_j$. By Assumption \ref{ass:reactive_support}, $Q_j^*$ is not located at the lower boundary of its feasible region when $v_j^* = \overline{v_j}$. Therefore, there exists a margin $\bar{\delta} > 0$ such that $Q_j^* - \bar{\delta}$ remains feasible.

We construct a joint perturbation: $l_{ij}' = l_{ij}^* - \epsilon$ and $Q_j' = Q_j^* - \delta$ with $\epsilon, \delta > 0$. To maintain the voltage constraint $v_j(l_{ij}', Q_j') = \overline{v_j}$, we use the affine expression above:
\begin{equation}
    v_j^* + \frac{\partial v_j}{\partial l_{ij}}(-\epsilon) + \frac{\partial v_j}{\partial Q_j}(-\delta) = \overline{v_j}.
\end{equation}
Since $v_j^* = \overline{v_j}$, solving for $\delta$ yields:
$
    \delta = \frac{z_{ij}}{2x_{ij}}\epsilon.
$

Choosing $\epsilon$ sufficiently small ensures that $\delta < \bar{\delta}$ and therefore preserves the feasibility of the reactive-power constraint. Define the function $g(\epsilon) = l_{ij}' v_i^* - P_{ij}'^2 - Q_{ij}'^2$, which represents the slack of the SOC constraint under perturbation, where $P_{ij}' = P_{ij}^* - r_{ij}\epsilon$ and $Q_{ij}' = Q_{ij}^* - x_{ij}\epsilon + \delta$. If the SOC constraint were not tight at the optimum, then $g(0) > 0$. By continuity of $g$, there exists a sufficiently small $\epsilon^* > 0$ such that $g(\epsilon) > 0$ for all $\epsilon \in [0, \epsilon^*]$.
Therefore, the perturbed solution $(l_{ij}^* - \epsilon^*, Q_j^* - \frac{z_{ij}}{2x_{ij}}\epsilon^*)$ remains feasible and achieves a smaller objective value, contradicting the optimality of $\mathcal{S}^*$. Hence, the SOC constraint must be tight.

\textbf{2. Inductive Step:}
Assume that for all downstream branches $(j,k)$, the constraints are tight, implying $l_{jk}^*$, $P_{jk}^*$, and $Q_{jk}^*$ are determined constants. Define the aggregate constants:
\begin{align}
    C_P &= P_j^* - \sum_{k:j \to k} P_{jk}^* + \sum_{k:j \to k} r_{jk}l_{jk}^*, \\
    C_Q &= Q_j^* - \sum_{k:j \to k} Q_{jk}^* + \sum_{k:j \to k} x_{jk}l_{jk}^*.
\end{align}
The power balance at node $j$ gives:
$P_{ij}^* = r_{ij}l_{ij}^* - C_P$ and $Q_{ij}^* = x_{ij}l_{ij}^* - C_Q$. Substituting into the SOC constraint yields a problem that is structurally identical to the leaf-branch case. Repeating the same argument shows that the constraint must also be tight for branch $(i,j)$.

By induction, the proposition holds for all $(i,j) \in \mathcal{E}$.
\end{proof}

\end{appendices}

\bibliographystyle{IEEEtran}
\bibliography{references}



\end{document}